\documentclass[11pt, a4paper, leqno]{amsart}
\usepackage{amsmath}
\usepackage{amssymb}
\usepackage{mathrsfs}
\usepackage{amsthm}  
\usepackage{microtype} 
\usepackage{enumitem}
\usepackage{xcolor}
\usepackage{microtype}
\usepackage{comment}
\usepackage{mathtools}
\usepackage[hidelinks]{hyperref}

\theoremstyle{definition}
\newtheorem{definition}{Definition}[section]
\newtheorem{theorem}[definition]{Theorem}
\newtheorem{corollary}[definition]{Corollary}
\newtheorem{lemma}[definition]{Lemma}
\newtheorem{proposition}[definition]{Proposition}
\newtheorem{question}[definition]{Question}
\newtheorem{fact}[definition]{Fact}

\newcommand{\restatethm}[2]{%
	\par\noindent
	\trivlist
	\item[\hskip \labelsep {\bfseries Theorem \ref{#1}.}]
	\normalfont 
	#2%
	\endtrivlist
}

\newcommand{\restatecor}[2]{%
	\par\noindent
	\trivlist
	\item[\hskip \labelsep {\bfseries Corollary \ref{#1}.}]
	\normalfont 
	#2%
	\endtrivlist
}

\newcommand{\restatelem}[2]{%
	\par\noindent
	\trivlist
	\item[\hskip \labelsep {\bfseries Lemma \ref{#1}.}]
	\normalfont 
	#2%
	\endtrivlist
}

\newcommand{\cL}{\mathcal L}

\newcommand{\sln}{\mathcal L^{\text{s},n}}

\newcommand{\crit}{\text{crit}}

\newcommand{\ZFC}{\text{ZFC}}

\newcommand{\HOD}{\text{HOD}}

\newcommand{\rk}{\text{rk}}

\newcommand{\Top}{\text{Top}}

\begin{document}
\title{Compactness beyond choice and HOD}
\author{Jonathan Osinski}
\address{The Czech Academy of Sciences, Institute of Computer Science}
\email{osinski@cs.cas.cz}
\date{\today}
\thanks{Work on this paper was supported by the Czech Science Foundation grant 25-16489S}
\thanks{I would like to thank Joan Bagaria for feedback on an earlier draft of this manuscript}
\begin{abstract}
	We characterise exacting, rank-Berkeley, Berkeley, and club Berkeley cardinals by compactness properties of logics and, equivalently, by compactness properties of certain topological spaces. This raises the strength known to be obtainable by statements about compactness into the realm of choiceless large cardinal axioms, the strongest known statements in terms of consistency strength. Moreover, it shows that compactness assumptions can directly violate the axiom of choice and the axiom $V=\HOD$.
\end{abstract}
\maketitle
\tableofcontents
	
	\section{Introduction}
	\noindent \emph{Compactness} is a major phenomenon across mathematical logic. The \emph{Compactness Theorem}, stating that any set of first-order sentences is satisfiable if and only if all its finite subsets are satisfiable, is perhaps the most important tool in model theory. In set theory, the existence of \emph{large cardinals} is intricately connected to compactness phenomena for extensions of first-order logics. Large cardinals play a major role for the foundations of mathematics. Axioms stating the existence of large cardinals are the yardstick for measuring the strength of strong mathematical assumptions. Moreover, assuming their existence can have consequences for other well-studied axioms, e.g., the axiom of choice and the axiom $V = \HOD$. In this article, we show that \emph{exacting}, \emph{rank-Berkeley}, \emph{Berkeley}, and \emph{club Berkeley} cardinals have characterisations by certain compactness properties of logics and, equivalently, by $[\kappa,\kappa]$-compactness, first studied by Alexandroff and Urysohn, of certain topological spaces (see Section \ref{sec:overview} for a statement of the results). This achieves three goals. First, it pushes up the known consistency strength achievable through compactness assertions into the realm of \emph{choiceless} large cardinals, and thus into the highest known regions of the consistency strength hierarchy. Second, it shows that certain compactness phenomena can lead to failures of the axiom of choice and of $V = \HOD$. And third, it establishes analogies between exacting and choiceless large cardinals on the one hand, and lower regions of the large cardinal hierarchy on the other, which also admit characterisations in terms of compactness. In this way, it connects these upper regions of the large cardinal hierarchy with compactness as studied in model theory and topology. To give some background which motivates our results, we discuss some of the relevance of these three goal  for ongoing discussions on the foundations of mathematics.
	
	\subsection{Compactness, consistency strength, and choiceless large cardinals}
	Large cardinals play a technical role in assessing the strength of mathematical theories. Due to Gödel's Second Incompleteness Theorem, foundational mathematical theories like ZFC cannot prove their own consistency. We thus order them by \emph{consistency strength}, where a theory $T$ is stronger than a theory $S$, if $T$ proves the consistency of $S$. Large cardinals are cofinal in the known parts of the consistency strength hierarchy and thus are used to measure a theory's strength. Kunen \cite{kunen1971elementary} showed that the existence of a non-trivial elementary embedding $j: V_{\lambda + 2} \to V_{\lambda + 2}$ for any ordinal $\lambda$ is inconsistent with $\ZFC$ (here $V_{\lambda + 2}$ is the collection of all sets of rank ${<} \lambda +2$; an embedding is \emph{elementary} if it preserves first-order truth). Among the strongest known axioms compatible with choice are the \emph{rank-into-rank embeddings} called I0--I3 in descending strength (cf. \cite{dimonte2018i0} for an overview). Note that Schlutzenberg showed equiconsistency between ZF plus the existence of an elementary $j: V_{\lambda + 2} \to V_{\lambda +2}$ and ZFC plus the existence of an I0 embedding (cf. \cite{schlutzenberg2025consistency}).
	
	Exceeding I0 in terms of consistency strength are the \emph{choiceless} large cardinal axioms which all imply the existence of Kunen's embedding inconsistent with AC (cf. \cite{bagaria2019large} for an overview). In particular, they thus contradict the axiom of choice. One example are \emph{Berkeley} cardinals. They were introduced by Woodin over 30 years ago in attempt to explore axioms possibly inconsistent with ZF, but to this day, no such inconsistency has arisen.\footnote{This origin story of Berkeley cardinals was reported in \cite{mohammd2026berkeley}.} \emph{Rank-Berkeley} cardinals were introduced by Schlutzenberg as a weakening of Berkeley cardinals. \emph{Club Berkeley} cardinals are a strengthening of Berkeleyness. In Sections \ref{sec:rank-Berkeley} and \ref{sec:Berkeley}, we characterise these choiceless cardinals by model-theoretic compactness properties. In Section \ref{sec:topology}, we derive topological compactness characterisations from this.
	
	From the point of view of the study of compactness, finding characterisations of cardinals with consistency strength as high as choiceless cardinals is conceptually interesting in the following way. Large cardinals are naturally connected to compactness properties of logics. A classical result of Magidor \cite{mag1971} shows that the smallest \emph{extendible cardinal} is the smallest cardinal $\kappa$ such that second-order logic $\cL^2$ is \emph{$\kappa$-compact}, i.e., the smallest extendible $\kappa$ is the smallest cardinal such that any $\cL^2$-theory is satisfiable if and only if all its ${<}\kappa$-sized subsets are satisfiable. A major dividing line for compactness principles in the large cardinal hierarchy is \emph{Vopěnka's Principle} (VP), which is an upper bound for the $\kappa$-compactness of strong logics:
	\begin{fact}[Makowsky \cite{mak1985}]\label{fact:makowsky}
		VP holds if and only if for every logic $\cL$ there is a cardinal $\kappa$ such that $\cL$ is $\kappa$-compact.
	\end{fact}
	Though both extendible cardinals and VP have considerable large cardinal strength, the rank-into-rank axioms and choiceless large cardinals significantly exceed both. In light of Makowsky's result, to find compactness properties at the highest ends of the consistency strength hierarchy above VP, it is not enough to consider ever stronger logics, we instead have to come up with new types of compactness properties. To our best knowledge, the highest consistency strength currently known to be achievable by compactness assertions lies just below I1 by a result of Boney using \emph{compactness for type omission properties} (cf. \cite[Theorem 4.16(2)]{bon2020}). Note that Boney does not give a relative consistency proof of the compactness principle he employs and so the question remains whether these principles are consistent relative to some established axioms. The strongest known equiconsistency result for compactness is the equivalence between certain compactness for type omission properties and the existence of \emph{$n$-huge cardinals}, which is weaker than I3 (cf. \cite[Theorem 3.11]{bon2020}). In any case, our characterisations considerably raise the consistency strength known to be obtainable by compactness properties, beyond I0 into the realm of choiceless large cardinals.
	
	Note that our results run counter to intuitions sometimes connected with compactness. It is well known that the Compactness Theorem from model theory is equivalent over ZF to the Boolean Prime Ideal Theorem (cf., e.g., \cite[Theorem 2.2]{jech2008axiom}), which is usually seen as a weak form of the axiom of choice. For instance, it implies that every set can be linearly ordered (cf., e.g., \cite[Theorem 2.18]{herrlich2006axiom}). In contrast, our results show that other forms of compactness can directly lead to failure of choice.
	
	\subsection{Compactness and exacting cardinals} \emph{Exacting} cardinals were recently introduced by Aguilera, Bagaria, and Lücke in \cite{aguilera2024large}. They showed that exacting cardinals cannot exist in $\HOD$, the universe of \emph{hereditarily ordinal definable sets} first studied by Gödel. Thus, the existence of an exacting cardinal implies that $V\neq \HOD$, i.e., that there are sets which are not definable from ordinal parameters. 
	While choiceless large cardinals also imply that $V\neq\HOD$, part of the importance of exacting cardinals stems from that they are compatible with AC. This contrasts them with all other traditionally studied large cardinals consistent with AC, which are all compatible with $V = \HOD$.
	Moreover, \cite{aguilera2024large} showed that the consistency of certain configurations of exacting cardinals alongside other large cardinals would refute Woodin's \emph{$\HOD$ conjecture}, one of the major open questions in set theory. 
	
	Exacting cardinals thus seriously challenge part of the picture of the large cardinal hierarchy previously drawn by set theorists (cf. \cite{aguilera2026large} for a discussion of the impact of exacting cardinals on the foundations of mathematics). Aguilera, Bagaria, Goldberg, and Lücke thus propose in \cite{aguilera2025large} to establish evidence whether exacting cardinals are natural candidates for additional mathematical axioms. They provide one such peace of evidence by characterising them in terms of \emph{(structural) reflection principles}, which is a type of set-theoretic property known to characterise large cardinals all over the hierarchy (cf. \cite{bagaria2023principles} for an overview).
	In this way, they provide some analogy to lower established levels of the large cardinal hierarchy. 
	
	In Section \ref{sec:exacting} we will characterise exacting cardinals in terms of model-theoretic compactness and in Section \ref{sec:topology} we will derive topological compactness characterisations from this. This reinforces the analogies of exacting cardinals to lower parts of the large cardinal hierarchy in which compactness is an ubiquitous  phenomenon. Next to the characterisations of extendible cardinals, VP, and $n$-huge cardinals mentioned above, many more of the most important large cardinals and large cardinal principles are currently known to have characterisations in terms of compactness properties of strong logics, among them: weakly compact cardinals (cf., e.g., \cite{kan}), measurable cardinals (cf., e.g., \cite{bon2020}), strong and $\Pi_n$-strong cardinals (\cite{bon2020, boney2025}), Woodin cardinals and \emph{Ord is Woodin} (\cite{boney2024model, boney2025}), Shelah cardinals (\cite{osinski2026model}), strongly compact cardinals (cf., e.g., \cite{kan}), supercompact cardinals (cf. \cite{bon2020, osinski2024Henkin}), and $C^{(n)}$-extendible cardinals (\cite{bon2020, osinski2024Henkin}).
	
	Characterising exacting cardinals in terms of compactness thus provides an additional analogy to other large cardinals, further reinforcing the idea that exacting cardinals fit naturally among traditional large cardinals. As an additional benefit, the connections to compactness provide some ``outside evidence" for exacting cardinals. While reflection is a set-theoretic phenomenon, compactness stems from topology and model theory, and so connects exacting cardinals with other parts of mathematics.

	\subsection{Overview of the results}\label{sec:overview}
	The key to our model-theoretic characterisations will be to consider compactness properties which are more specific than usual about the shape of models they provide. The properties we consider derive the existence of a model of a set of sentences $T$ of a specific \emph{rank type}, \emph{size type}, and \emph{relation type} from the existence of models of subsets of $T$ with the same type. Here, the first two respectively fix the rank or the size of a specified unary relation on the model, while the latter fixes the isomorphism type of a binary relation on the model. Rank types were also used to prove the characterisations of exacting cardinals in terms of structural reflection in \cite{aguilera2025large}. We introduce variations of so-called \emph{chain compactness} in which satisfiability of a theory filtrated as an increasing union $T = \bigcup_{\alpha < \mu} T_\alpha$ is derived from satisfiability of all $T_\alpha$. Alternatively, all results mentioned can also be proven as variants of so-called \emph{$[\kappa,\mu]$-compactness} first studied by Makowsky and Shelah \cite{makowsky1979theorems, makowsky1983positive}. 
	
	First, in Section \ref{sec:preliminaries}, we fix some terminology from model theory of extensions of first-order logic, in particular adapted to the choiceless setting. In Section \ref{sec:rank-Berkeley}, we will introduce \emph{$\mu$-chain compactness for rank types} and show our results on rank-Berkeley cardinals. We first characterise proto rank-Berkeley cardinals (Theorem \ref{thm:proto-rank}), and derive:
	\restatecor{cor:existence-rank-Berkeley}{
		There is a rank-Berkeley cardinal if and only if there is a cardinal $\lambda$ such that for all $\zeta > \lambda$ there is a $\mu < \lambda$ such that $\cL^2$ is $\mu$-chain compact for rank type $(\zeta, \lambda)$.
	}
	\restatethm{thm:rank-Berkeley}{
		The following are equivalent for a cardinal $\lambda$.
		\begin{enumerate}
			\item[(1)] $\lambda$ is rank-Berkeley.
			\item[(2)] For every $\alpha < \lambda < \zeta$ there is $\alpha < \mu < \lambda$ such that $\cL^2_{\mu \omega}$ is $\mu$-chain compact for rank type $(\zeta, \lambda)$.
		\end{enumerate}
	}
	The full version of Theorem \ref{thm:rank-Berkeley} also contains additional characterisations in terms of stronger logics. 
	
	In Section \ref{sec:Berkeley}, we introduce \emph{$\mu$-chain compactness for relation types} and show our results on Berkeley cardinals. We characterise proto Berkeley cardinals (Theorem \ref{thm:proto-Berkeley}), and derive:
	\restatecor{cor:existence-Berkeley}{
		There is a Berkeley cardinal if and only if there is a cardinal $\lambda$ such that for every relation $(M,R)$ there is some $\mu < \lambda$ such that first-order logic is $\mu$-chain compact for relation type $(M,R)$.
	}
	Note that this shows that the existence of a Berkeley cardinal is equivalent to a compactness property of \emph{first-order} logic.
	\restatethm{thm:Berkeley}{
		The following are equivalent for a cardinal $\lambda$.
		\begin{enumerate}
			\item[(1)] $\lambda$ is Berkeley.
			\item[(2)] For every $\alpha < \lambda$ and every relation $(M,R)$ there is a cardinal $\alpha < \mu < \lambda$ such that $\cL_{\mu \omega}$ is $\mu$-chain compact for relation type $(M,R)$.
		\end{enumerate}
	}
	Again, the full version of Theorem \ref{thm:Berkeley} also contains additional characterisations in terms of stronger logics. Asserting for $\mu$-chain compactness cardinals for relation types to exist in clubs of $\lambda$ leads to a characterisation of club Berkeleyness (Theorem \ref{thm:club-Berkeley})
	
	In Section \ref{sec:exacting}, we introduce a further variant called \emph{$(\mu,\lambda)$-chain compactness for rank types} and use it to prove compactness characterisations of exacting cardinals. The only difference is that the size of the theory considered is additionally restricted. To prove our results on exacting cardinals, we first show that a seemingly stronger form of the definition of exactingness is actually equivalent to exactingness:
	\restatelem{lem:multiexacting}{
		Let $\lambda$ be a cardinal. The following are equivalent:
		\begin{enumerate}
			\item[(1)] $\lambda$ is exacting.
			\item[(2)] For every $\alpha < \lambda < \zeta$ and every elementary substructure $X \prec V_\zeta$ with $|X| = \lambda$ and $V_\lambda \cup \{\lambda\} \subseteq V_\zeta$ there is an elementary embedding $j: X \to V_\zeta$ with $\alpha < \crit(j) < \lambda$ and $j(\lambda) = \lambda$.
		\end{enumerate}
	}
	We use this to derive our compactness characterisation:
	\restatethm{thm:exacting-compactness}{
		The following are equivalent for $\lambda = \beth_\lambda$:
		\begin{enumerate}
			\item[(a)] $\lambda$ is exacting. 
			\item[(b)] For every $\zeta \geq \lambda$ there is a cardinal $\mu < \lambda$ such that $\cL^2$ is $(\mu,\lambda)$-chain compact for rank type $(\zeta, \lambda)$. 
		\end{enumerate}
	}
	Again, the full version of Theorem \ref{thm:exacting-compactness} contains additional characterisations in terms of stronger logics. We also use the compactness results to derive even stronger embedding characterisations of exactingness (Corollaries \ref{cor:char-exacting} and \ref{cor:char-exacting2}), which in some sense are close to inconsistency (Proposition \ref{prop:exacting-inconsistent}), and derive a characterisation of the existence of a proper class of exacting cardinals reminiscent of Makowsky's result on VP (Corollary \ref{cor:class-exacting}).
	
	As a corollary to the model-theoretic characterisations, in Section \ref{sec:topology} we derive characterisations of exacting and choiceless large cardinals by compactness properties of certain topological spaces. The compactness property we will use is also called \emph{$[\kappa,\mu]$-compactness} and was introduced by Alexandroff and Urysohn \cite{aleksandrov1929memoire} (as cited in \cite[p. 649]{bar1985}). It was observed by Mannila \cite{mannila1983topological} that the logical definition of $[\kappa,\mu]$-compactness of a logic is equivalent to $[\kappa,\mu]$-compactness of spaces of structures in the topological sense. The same is true for the compactness variants we introduced and so the logical compactness characterisations of large cardinals in the article can be reformulated to characterisations in terms of $[\kappa,\mu]$-compactness of topological spaces. We list some examples as Corollaries \ref{cor:rank-Berkeley-topology} to \ref{cor:exacting-topology}. Note that in contrast to the logical results, the topological results do \emph{not} change the compactness property studied but use the original notion by Alexandroff and Urysohn. The large cardinal strength instead comes from varying the topological spaces $[\kappa,\mu]$-compactness applies to.

	\section{Preliminaries}\label{sec:preliminaries}
	Our default is to work in $\ZFC$, though in Sections \ref{sec:rank-Berkeley} and \ref{sec:Berkeley} we work in ZF to accommodate choiceless large cardinals. If $X$ is a well-orderable set, we denote by $|X|$ the unique cardinal in bijection to $X$. In the ZF context, sets may not have a cardinality in this sense. If we write $|X|$ we thus implicitly assume that $X$ is well-orderable. If $\mu$ is a cardinal and $X$ any set, we write $\mathcal P_\mu X = \{Y \subseteq X \colon |Y| < \mu \}$. 
	
	We mention some conventions from model theory (of extensions of first-order logic). A \emph{vocabulary} $\tau$ is a set consisting of relation and function symbols, each of a finite arity $n \in \omega$, as well as of constant symbols. A $\tau$-structure $A$ is a set (the \emph{universe} of $A$ which we also denote by $A$) equipped with relations $R^A \subseteq A^n$ for each $n$-ary relation symbol $R \in \tau$, functions $f^A : A^m \to A$ for each $m$-ary function symbol $f \in \tau$, and $c^A \in A$ for each constant symbol $c \in \tau$. A \emph{logic} $\cL$ consists of a definable map $\tau \mapsto \cL[\tau]$ which associates some set $\cL[\tau]$ with every vocabulary $\tau$ (the \emph{set of $\cL$-sentences over $\tau$}), and a definable relation $\models_\cL$ such that if 
	\[
	A \models_\cL \varphi,
	\]
	then $A$ is a $\tau$-structure for some vocabulary $\tau$ and $\varphi \in \cL[\tau]$ is an $\cL$-sentence. The relation $\models_\cL$ is also called \emph{satisfaction relation of $\cL$}. We also say that $A$ \emph{is a model of} $\varphi$, or \emph{satisfies} $\varphi$. It is standard to assume some intuitive properties of $\models_\cL$. For instance, if $A$ and $B$ are isomorphic $\tau$-structures, then we assume for all $\varphi \in \cL[\tau]$, we have $A \models_\cL \varphi$ iff $B \models_\cL \varphi$. For a standard, comprehensive definition, we refer the reader to \cite[Chapter II, Definition 1.1.1]{bar1985}.
	We assume that our logics are closed under boolean connectives, i.e., for every $\varphi, \psi \in \cL[\tau]$ there are sentences $\neg \varphi, \varphi \wedge \psi \in \cL[\tau]$ such that $A \models_\cL \neg\varphi$ iff $A \not \models_\cL \varphi$, and $A \models_\cL \varphi \wedge \psi$ iff $A \models_\cL \varphi$ and $A \models_\cL \psi$. 
	
	We can compare logics in their expressive strength. For this, we write $\cL \leq \cL^*$ if for every vocabulary $\tau$ and every $\varphi \in \cL[\tau]$ there is some $\psi \in \cL^*[\tau]$ such that for all $\tau$-structures $A$:
	\[
	A \models_\cL \varphi \text{ if and only if } A \models_{\cL^*} \psi.
	\]
	First-order logic is denoted by $\cL_{\omega \omega}$. Because we are interested in extensions of first-order logic, we assume that $\cL_{\omega \omega} \leq \cL$.
	
	If $\tau$ and $\sigma$ are vocabularies, a bijection $f: \tau \to \sigma$ is called a \emph{renaming}, if it is an isomorphism of vocabularies, i.e., $f$ restricts to bijections of the respective sets of relation, function, and constant symbols, all while preserving arities. The renaming $f$ also induces a \emph{renaming of structures}, turning a $\tau$-structure $A$ into a $\sigma$-structure $f(A)$ with the same universe by letting $f(R)^{f(A)} = R^A$ for $R \in \tau$. Two sentences $\varphi \in \cL[\tau]$ and $\psi \in \cL[\sigma]$ are called \emph{equivalent up to renaming} if for all $\tau$-structures $A$:
	\[
	A \models_\cL \varphi \text{ iff } f(A) \models_\cL \psi. 
	\]
	
	When writing $T \subseteq \cL$ we mean that $T \subseteq \cL[\tau]$ for some vocabulary $\tau$. Such a $T$ is also called an \emph{$\cL$-theory}. If $A \models_\cL \varphi$ for all $\varphi \in T$, we write $A \models_\cL T$ and say that $A$ \emph{is a model of $T$}. If there is a model of $T$, then $T$ is  \emph{satisfiable}. We will consider variants the following compactness property introduced by Makowsky and Shelah \cite{makowsky1979theorems, makowsky1983positive}.
	\begin{definition}
		Let $\cL$ be a logic. For cardinals $\kappa$ and $\mu$, we say that $\cL$ is \emph{$[\kappa, \mu]$-compact} if for all theories $T, S \subseteq \cL$ with $|S| = \kappa$, if for every $S_0 \in \mathcal P_\mu S$ there is a model of $T \cup S_0$, then there is a model of $T \cup S$. 
	\end{definition}
	It is known that $[\kappa, \mu]$-compactness is connected to satisfiability of theories which can be written as increasing unions, where each small bit of the union is satisfiable (cf. \cite{mannila1983topological}). In the contexts we are interested in, this reduces to so-called \emph{chain-compactness properties}.
	\begin{lemma}[ZF]\label{lem:square-chain}
		Let $\cL$ be a logic and $\mu$ a regular cardinal. Then $\cL$ is $[\mu,\mu]$-compact if and only if $\cL$ is \emph{$\mu$-chain compact}, i.e., every theory $T \subseteq \cL$ is satisfiable, if it can be written as an increasing union $T = \bigcup_{\alpha < \mu} T_\alpha$ such that every $T_\alpha$ is satisfiable. 
			\begin{proof}
				The proof is an easy variation of a more general statement in \cite{mannila1983topological}. For the forward direction, suppose we have $T = \bigcup_{\alpha < \mu} T_\alpha$ where each $T_\alpha$ has a model. Take distinct predicates $P_\alpha$ not appearing in the language of $T$, let $R = \{\exists x P_\alpha(x) \rightarrow \varphi \colon \alpha < \mu, \varphi \in T_\alpha \}$, and $S = \{\exists x P_\alpha(x) \colon \alpha < \mu\}$. If $S_0 \in \mathcal P_\mu S$, since $\mu$ is regular there is $\beta < \mu$ such that $S_0 \subseteq \{\exists x P_\alpha(x) \colon \alpha < \beta \}$. Then if $A \models T_\beta$, we get that $(A, P^A_\alpha)_{\alpha < \mu} \models R \cup S_0$ where for $\alpha < \beta$, $P_\alpha^A$ is any non-empty subset of $A$, and for $\alpha \geq \beta$, $P^A_\alpha = \emptyset$. By $[\mu,\mu]$-compactness, it follows that $R \cup S$ has a model $B$. Clearly, $B$ is also a model of $T$. 
				
				\medskip
				
				For the backwards direction, let $T, S\subseteq \cL$ with $|S| = \mu$, and with models of $T \cup S_0$ for any $S_0 \in \mathcal P_\mu S$. List $S = \{\varphi_\alpha \colon \alpha < \mu\}$. For $\alpha < \mu$, let $R_\alpha = T \cup \{\varphi_\beta \colon \beta < \alpha\}$. Because $\{\varphi_\beta \colon \beta < \alpha \} \in \mathcal P_\mu S$, we get that $R_\alpha$ has a model. By assumption, thus $\bigcup_{\alpha< \mu} R_\alpha = T \cup S$ has a model. 
			\end{proof}
	\end{lemma}
	For cardinals $\kappa,\lambda$, we let $\cL_{\kappa \lambda}$ be the extension of first-order logic by conjunctions and disjunctions over sets $T$ of formulas with $|T| < \kappa$, and by first-order quantifiers over sets of variables of cardinality ${<}\lambda$. Similarly, $\cL^2_{\kappa \lambda}$ extends second-order logic. We will use the following folklore fact.
	\begin{fact}\label{fact:definition}
		Let $\kappa$ be a cardinal. Then for every $a \in H_\kappa$, the logic $\cL_{\kappa \omega}$ can define $a$, i.e., there is a formula $\sigma_a(x) \in \cL_{\kappa \omega}$ such that for every transitive set $M$ and every $b \in M$:
		\[
		(M, \in) \models \sigma_a(b) \text{ if and only if } a=b. 
		\]
		\begin{proof}
			By $\in$-induction. For $a = \emptyset$, let $\sigma_a(x) = \neg \exists y(y \in x)$. If the statement holds for all members of $a$, and so we have the formulas $\sigma_c$ for $c \in a$, let $\sigma_a(x) = \forall y(y \in x \leftrightarrow \bigvee_{c \in a} \sigma_c(y))$. Because $a \in H_\kappa$, we get $\sigma_a(x) \in \cL_{\kappa \omega}$. 
		\end{proof}
	\end{fact}
	The proof does not work in ZF, as the considered disjunction is indexed by $a$, which does not lead to a formula of $\cL_{\kappa \omega}$ if $a$ is not well-orderable. Still, the result holds in ZF for all ordinals $\alpha < \kappa$.
	
	We denote second-order logic by $\cL^2$. Recall the following result by Magidor.
	\begin{fact}[Magidor \cite{mag1971}]\label{fact:Magidor}
		There is a sentence $\Phi \in \cL^2$ in the language $\{\in\}$ of set theory, known as \emph{Magidor's $\Phi$}, such that
		\[
		\begin{split}
		(M,E) \models \Phi \text{ if and only if} \text{ there is a limit ordinal } \alpha \text{ with } (M,E) \cong (V_\alpha, \in).
		\end{split}  
		\]
	\end{fact}
	The construction can be adapted to get a sentence axiomatising the class of structures isomorphic to any $V_\alpha$, not just those where $\alpha$ is a limit ordinal (cf., e.g., \cite[Lemma 1.2.4]{osinskiPhD}). Both constructions are available in ZF.
	
	We will also consider \emph{sort logics}, an extension of second-order logic introduced by Väänänen \cite{vää1979}. The results on sort logic intend to give an as complete picture as possible for a reader interested in abstract model theory. But all large cardinals considered in the article can already be characterised by weaker logics of strength at most infinitary second-order, so the statements on sort logic may be skipped. Sort logics are graded by the natural numbers $n$ into logics $\sln$ of ascending strength. Again, we denote by $\sln_{\kappa \omega}$ the extension of $\sln$ by conjunctions and disjunctions over sets $T$ of formulas with $|T| < \kappa$. The precise workings of sort logics will not matter in this article (cf. \cite{vää2014} for an introduction). For us, the important fact about sort logics is that they are the strongest logics available in the sense that for every logic $\cL$ there is some natural number $n$ and some cardinal $\kappa$ such that $\cL \leq  \sln_{\kappa \omega}$ (the statement is implicit in \cite[Section 3]{vää2014}; for a full proof cf. \cite[Corollary 1.2.24]{osinskiPhD}; we do not know whether the statement holds in ZF).
	
	\section{Rank-Berkeley cardinals and compactness for rank type}\label{sec:rank-Berkeley}
	In this and the following section, we assume ZF throughout. For an ordinal $\alpha$, a cardinal $\lambda$ is \emph{$\alpha$-proto rank-Berkeley} if for all $\zeta > \lambda$ there is an elementary $j: V_\zeta \to V_\zeta$ with $j(\lambda) = \lambda$ and $\alpha < \crit(j) < \lambda$. It is called \emph{rank-Berkeley} if for all $\alpha < \lambda < \zeta$ there is an elementary $j: V_\zeta \to V_\zeta$ with $j(\lambda) = \lambda$ and $\alpha < \crit(j) < \lambda$.\footnote{Sometimes the definition of a rank-Berkeley additionally includes the requirement that $\lambda = \kappa_\omega(j)$ (cf., e.g., \cite{goldberg2023periodicity}).} We will use the following result:
	\begin{fact}[{Mohammd \cite[Proposition 4.8]{mohammd2026berkeley}}]\label{fact:proto-rank}
		Let $\alpha$ be an ordinal. The smallest $\alpha$-proto rank-Berkeley cardinal, if it exists, is also rank-Berkeley.
	\end{fact} 
 	
 	To characterise rank-Berkeley, Berkeley, and exacting cardinals by compactness properties, we consider compactness properties which specify more precisely which models they produce. If $\alpha$ and $\beta$ are ordinals, and $\tau$ is a vocabulary with a designated unary relation symbol $P$, we say that a $\tau$-structure $A$ has \emph{rank type $(\alpha,\beta)$} if $\rk(A) = \alpha$ and $\rk(A^P) = \beta$. Rank types were also used to characterise exacting cardinals by structural reflection properties  (cf. \cite[Section 6]{aguilera2025large}). We introduce variants of $\mu$-chain compactness and $[\kappa,\mu]$-compactness which restrict to models of a specified rank type.
 	
 	\begin{definition}
 		Fix a logic $\cL$, cardinals $\kappa$ and $\mu$, and ordinals $\alpha$ and $\beta$.
 		\begin{enumerate}
 			\item[(1)] $\cL$ is \emph{$\mu$-chain compact for rank type $(\alpha, \beta)$} if for every theory $T \subseteq \cL$ which can be written as an increasing union $T = \bigcup_{\alpha < \mu} T_\alpha$, if every $T_\alpha$ has a model of rank type $(\alpha, \beta)$, then also $T$ has a model of rank type $(\alpha, \beta)$.
 			\item[(2)] $\cL$ is \emph{$[\kappa, \mu]$-compact for rank type $(\alpha, \beta)$} if for all theories $T,S \subseteq \cL$ with $|S| = \kappa$, if for every $S_0 \in \mathcal P_\mu S$ there is a model of $T \cup S_0$ of rank type $(\alpha, \beta)$, then there is also a model of $T \cup S$ of rank type $(\alpha, \beta)$.
 		\end{enumerate}
 	\end{definition}

The equivalence between $\mu$-chain compactness and $[\mu,\mu]$-compactness carries over:
\begin{lemma}\label{lem:square-chain-type}
	Let $\cL$ be a logic and $\mu$ a regular cardinal. Then $\cL$ is $[\mu, \mu]$-compact for rank type $(\alpha, \beta)$ iff $\cL$ is $\mu$-chain compact for rank type $(\alpha, \beta)$.
	\begin{proof}
		The same argument as that for Lemma \ref{lem:square-chain} works simply by letting all models considered have rank type $(\alpha, \beta)$.
	\end{proof}
\end{lemma}

\begin{theorem}\label{thm:proto-rank}
	The following are equivalent for a cardinal $\lambda$:
	\begin{enumerate}
		\item[(1)] $\lambda$ is $0$-proto rank-Berkeley, i.e., for all $\zeta > \lambda$ there is an elemenetary $j: V_{\zeta} \to V_\zeta$ with $j(\lambda) = \lambda$ and $\crit(j) < \lambda$. 
		\item[(2)] For all $\zeta > \lambda$ there is a $\mu < \lambda$ such that $\cL^2$ is $\mu$-chain compact for rank-type $(\zeta, \lambda)$.
	\end{enumerate}
	\begin{proof}
		Assume (1) and suppose towards a contradiction that (2) fails, and thus that there is a smallest $\zeta > \lambda$ such that for all $\mu < \lambda$ there is a theory $T = \bigcup_{\beta < \mu} T_\beta \subseteq \cL^2$ such that all $T_\beta$ have models of type $(\zeta, \lambda)$ but $T$ has no such model. Let $\eta > \zeta$ reflect that $\zeta$ is the smallest ordinal $> \lambda$ with the above property. In particular, for every $\mu < \lambda$, $V_\eta$ contains such a counterexample theory $T = \bigcup_{\beta < \mu} T_\beta \subseteq \cL^2$ along with models of rank type $(\zeta, \lambda)$ for every $T_\beta$. By $0$-proto rank-Berkeleyness of $\lambda$, let $j: V_\eta \to V_\eta$ with $\crit(j) < \lambda = j(\lambda)$. Note that $\zeta$ is definable from parameter $\lambda$ and so because $\lambda$ is fixed by $j$, and since $V_\eta$ reflects the definition, we get $j(\zeta) = \zeta$. Let $\mu = \crit(j)$. Consider the corresponding counterexample $T = \bigcup_{\beta < \mu} T_\beta$ in $V_\eta$. Then $V_\eta$ verifies that each $T_\beta$ has a model of type $(\zeta,\lambda)$, but $T$ has no such model. By elementarity, for some sequence of theories $(T_\beta^* \colon \beta < j(\mu))$, we get that $V_\zeta$ satisfies:
		\begin{quote}
			``$j(T) = \bigcup_{\beta < j(\mu)} T_\beta^*$ is an increasing union and every $T_\beta^*$ has a model of rank type $(j(\zeta), j(\lambda)) = (\zeta, \lambda)$, but $j(T)$ has no such model."
		\end{quote}
		Note that $j``T = \bigcup_{\beta < \mu} j``T_\beta \subseteq \bigcup_{\beta < \mu} j(T_\beta) \subseteq T_\mu^*$. Since $\mu < j(\mu)$, we get that $V_\eta$ has a model $A$ of $T_\mu^*$ and thus of $j``T$ of rank type $(\zeta, \lambda)$. Because $V_\eta$ is correct about $\cL^2$-satisfaction, $A$ is really a model of $j``T$. It is easy to check using elementarity of $j$, that $j: \tau \to j``\tau$ is a renaming and that the set $j``T$ is really an $\cL^2[j``\tau]$-theory. Then we can rename $A$ to a $\tau$-structure $B$ of rank type $(\zeta, \lambda)$, and for $\varphi \in T$, we have that $\varphi$ and $j(\varphi)$ are equivalent up to renaming, so $B \models \varphi$ iff $A \models j(\varphi)$. Then $B$ is really a model of $T$ of rank type $(\zeta, \lambda)$. But we assumed that there is no such model. Contradiction.
		
		\medskip
		
		And now assume (2) and let us show (1). Let $\zeta > \lambda$. We seek to find an elementary $j: V_\zeta \to V_\zeta$ with $\crit(j) < \lambda$ and $j(\lambda) = \lambda$. Take a $\mu$ corresponding to $\zeta$ according to (2). Consider the theory:
		\[\begin{split}
			T = & \text{ ElDiag}(V_\zeta, \in) \cup \{\Phi \} \cup \{c_i < c \wedge c < c_\mu \colon i < \mu \} \\ & \cup \{\forall x(P(x) \leftrightarrow \rk(x) < c_\lambda) \},
		\end{split}
		\]
		where $c$ is a constant symbol not appearing the elementary diagram and $P$ is a unary relation symbol. We can write $T = \bigcup_{\beta < \mu} T_\beta$ by letting for $\beta < \mu$:
		\[\begin{split}
			T_\beta = & \text{ ElDiag}(V_\zeta, \in) \cup \{\Phi \} \cup \{c_i < c \wedge c < c_\mu \colon i < \beta \}  \\ & \cup \{\forall x(P(x) \leftrightarrow \rk(x) < c_\lambda) \}.
		\end{split}
		\]
		Then $T_\beta$ is satisfiable by a model of rank type $(\zeta, \lambda)$, namely by \[A = (V_\zeta, \in, P^A, c^A, c_x^A)_{x \in V_\zeta}\] letting $P^A = V_\lambda$, $c^A = \beta$, and $c_x^A = x$. Thus, by (2), $T$ has a model $B$ of rank type $(\zeta, \lambda)$. We have $j: V_\zeta \to B$ elementary with $x \mapsto c_x^B$. Because $B$ satisfies Magidor's $\Phi$ (cf. Fact \ref{fact:Magidor}), without loss of generality $B = V_\eta$ for some $\eta$. Because $\rk(B) = \zeta$, thus $B = V_\zeta$ and $j: V_\zeta \to V_\zeta$. By $T$, $P^B$ contains exactly the sets of rank $< c_\lambda^B$, so $P^B = V_{c_\lambda^B}$. Because $\rk(P^B) = \lambda$, it follows that $V_{c_\lambda^B} = V_\lambda$. Thus, $c_\lambda^B = \lambda$, and hence $j(\lambda) = c_\lambda^B = \lambda$. Further, $j(\mu) = c_\mu^B > c^B$ and $c^B$ has order type at least $\mu$. Thus $j(\mu) > \mu$ and $j$ has a critical point $\crit(j) \leq \mu < \lambda$. 
	\end{proof}
\end{theorem}

Note that the direction from (1) to (2) shows that in (2), $\mu$ can be chosen to be a regular cardinal, because the contradiction arises from $\mu<\lambda$ which is the critical point of $j$, and critical points are regular. Using Lemma \ref{lem:square-chain-type}, we can therefore also consider $[\mu,\mu]$-compactness for rank type $(\zeta,\lambda)$ in (2). The same is true for the other results throughout.

\begin{corollary}\label{cor:existence-rank-Berkeley}
	There is a rank-Berkeley cardinal if and only if there is a cardinal $\lambda$ such that for all $\zeta > \lambda$ there is a $\mu < \lambda$ such that $\cL^2$ is $\mu$-chain compact for rank-type $(\zeta, \lambda)$.
	\begin{proof}
		If $\lambda$ is rank-Berkeley it is also $0$-proto rank-Berkeley and then by Theorem \ref{thm:proto-rank}, $\lambda$ witnesses the second part of the statement. For the converse, if $\lambda$ is as in the conclusion, then it is $0$-proto rank-Berkeley by Theorem \ref{thm:proto-rank}. Then by Fact \ref{fact:proto-rank}, the smallest $0$-proto rank-Berkeley cardinal is also rank-Berkeley.
	\end{proof}
\end{corollary}

The following fully characterises rank-Berkeleyness.

\begin{theorem}\label{thm:rank-Berkeley}
	Let $\lambda$ be a cardinal. The following are equivalent:
	\begin{enumerate}
		\item[(1)] $\lambda$ is rank-Berkeley.
		\item[(2)] For every $\alpha < \lambda < \zeta$ there is a cardinal $\alpha < \mu < \lambda$ such that $\cL^2_{\mu \omega}$ is $\mu$-chain compact for rank type $(\zeta, \lambda)$.
		\item[(3)] For every natural number $n$ and every $\alpha < \lambda < \zeta$ there is a cardinal $\alpha < \mu < \lambda$ such that $\sln_{\mu \omega}$ is $\mu$-chain compact for rank type $(\zeta,\lambda)$. 
	\end{enumerate}
	\begin{proof}
		Clearly (3) implies (2). For (1) implies (3), assume (1), fix a natural number $n$, and let $\alpha < \lambda$. The proof is similar to the one of Theorem \ref{thm:proto-rank}, using that we can choose a $V_\eta$ which computes $\sln_{\mu \omega}$-satisfaction correctly. Towards a contradiction, suppose that there is a smallest $\zeta \geq \lambda$ such that for all $\mu$ with $\alpha < \mu < \lambda$ there is a theory $T = \bigcup_{\beta < \mu} T_\beta \subseteq \sln_{\mu \omega}$ such that all $T_\beta$ have models of rank type $(\zeta, \lambda)$ but $T$ has no such model. Take an $\eta > \zeta$ which is correct about $\sln_{\mu \omega}$-satisfaction for every $\mu$ and which reflects the above definition of $\zeta$ from $\alpha$ and $\lambda$. Then for every $\mu < \lambda$, $V_\eta$ contains a counterexample theory $T = \bigcup_{\beta < \mu} T_\beta \subseteq \sln_{\mu \omega}$ alongside models of rank type $(\zeta, \lambda)$ for every $T_\beta$. By (1), let $j: V_\eta \to V_\eta$ with $\alpha < \crit(j) < \lambda = j(\lambda)$. Note that $\zeta$ is definable from parameters $\lambda$ and $\alpha$ and so because $\lambda$ and $\alpha$ are both fixed by $j$ and because $V_\eta$ reflects the definition, we get $j(\zeta) = \zeta$. Let $\mu = \crit(j)$. Consider the corresponding counterexample $T = \bigcup_{\beta < \mu} T_\beta$ in $V_\eta$. Then $V_\eta$ verifies that each $T_\beta$ has a model of rank type $(\zeta,\lambda)$, but $T$ has no such model by its correctness. By elementarity, $V_\eta$ believes that $j(T) = \bigcup_{\beta < j(\mu)} T_\beta^*$ for some theories $T_\beta^*$ where each $T_\beta^*$ has a model of rank type $(j(\zeta),j(\lambda)) = (\zeta,\lambda)$. Then $V_\eta$ believes that $T_\mu^*$ and thus $j``T \subseteq T_\mu^*$ has a model $A$ of rank type $(\zeta, \lambda)$ and is correct about this by its correctness. Because $\crit(j) = \mu$, applying $j$ to an $\sln_{\mu \omega}$-sentence $\varphi$ leads to a sentence $j(\varphi)$ which is equivalent up to renaming. Then $A$ can be renamed to a model of $T$ of rank type $(\zeta, \lambda)$ along the renaming $j: \tau \to j``\tau$, contradicting that there is no such model.
		
		\medskip

		Assume (2), let $\alpha < \lambda$ and $\zeta > \lambda$. We seek to find $j: V_\zeta \to V_\zeta$ with $\alpha < \crit(j) < \lambda$ and $j(\lambda) = \lambda$. Take $\mu$ with $\alpha < \mu < \lambda$ according to (2). Consider the theory:
		\[
		\begin{split}
			T = & \text{ ElDiag}(V_\zeta, \in) \cup \{\Phi \}  \cup \{c_i \in c \wedge c < c_\mu \colon i < \mu \}  
			\\ & \cup \{\forall x(P(x) \leftrightarrow \rk(x) < c_\lambda) \} \cup \{\sigma_i(c_i) \colon i < \mu \}.
		\end{split}
		\]
		The only difference to the theory in the proof of Theorem \ref{thm:proto-rank} is that we include definitions of the ordinals $< \mu$ (recall Fact \ref{fact:definition}). We can write $T = \bigcup_{\beta < \mu} T_\beta$ by letting for $\beta < \mu$:
		\[
		\begin{split}
			T_\beta = & \text{ ElDiag}(V_\zeta, \in) \cup \{\Phi \}  \cup \{c_i \in c \wedge c < c_\mu \colon i < \beta \}  \\ & \cup \{\forall x(P(x) \leftrightarrow \rk(x) < c_\lambda) \} \cup \{\sigma_i(c_i) \colon i < \mu \}.
		\end{split}
		\]
		Then $T_\beta$ is satisfiable by a model $A$ of rank type $(\zeta, \lambda)$ expanding $V_\zeta$ letting $c^A=\beta$. Thus, by (2), $T$ has a model $B$ of rank type $(\zeta, \lambda)$. Exactly as before we get an elementary $j: V_\zeta \to V_\zeta$ with $j(\lambda) = \lambda$ and $\crit(j) \leq \mu$. Every ordinal $i < \mu$ is such that $\sigma_i(c_i^B)$ holds and so is fixed by $j$. Thus $\crit(j) = \mu > \alpha$.
	\end{proof}
\end{theorem}

\section{Berkeley cardinals and compactness for relation type}\label{sec:Berkeley}
\subsection{Berkeley cardinals} For an ordinal $\alpha < \lambda$, a cardinal $\lambda$ is \emph{$\alpha$-proto Berkeley} if for all transitive $M$ with $\lambda \in M$ there is an elementary $j: M \to M$ with $\alpha < \crit(j) < \lambda$. It is called \emph{Berkeley} if for all $\alpha < \lambda$ and all transitive $M$ with $\lambda \in M$ there is an elementary $j: M \to M$ with $\alpha < \crit(j) < \lambda$ (cf. \cite{bagaria2019large}).
\begin{fact}[{Bagaria, Koellner, Woodin \cite[Remark 3.9]{bagaria2019large}}]\label{fact:proto-Berkeley}
	Let $\alpha$ be an ordinal. The smallest $\alpha$-proto Berkeley cardinal, if it exists, is also Berkeley.
\end{fact} 
\begin{fact}[{Bagaria, Koellner, Woodin \cite[Lemma 3.4]{bagaria2019large}}]\label{fact:proto-fix}
	A cardinal $\lambda$ is $\alpha$-proto Berkeley if and only if for all sets $a$ and all transitive $M$ with $\lambda, a \in M$ there is an elementary $j: M \to M$ such that $\alpha < \crit(j) < \lambda$ and $j(a) = a$.
\end{fact}
In the compactness properties witnessing Berkeleyness,  we restrict to structures which have a designated relation in the following way. We call a pair $(M,R^M)$ a \emph{relation} if  $M$ is a set and $R^M \subseteq M^2$. Let $\tau$ be a vocabulary with a designated binary relation symbol $R$. For a relation $(M, R^M)$, we say that a $\tau$-structure $A$ has \emph{relation type $(M,R^M)$} if $(A,R^A) \cong (M, R^M)$.

\begin{definition}
	Fix a logic $\cL$, cardinals $\kappa$ and $\mu$, and a relation $(M,R)$.
	\begin{enumerate}
		\item[(1)] $\cL$ is \emph{$\mu$-chain compact for relation type $(M,R)$} if for every theory $T \subseteq \cL$ which can be written as an increasing union $T = \bigcup_{\alpha < \mu} T_\alpha$, if every $T_\alpha$ has a model of relation type $(M,R)$, then also $T$ has a model of relation type $(M,R)$.
		\item[(2)] $\cL$ is \emph{$[\kappa, \mu]$-compact for relation type $(M,R)$} if for all theories $T,S \subseteq \cL$ with $|S| = \kappa$, if for every $S_0 \in \mathcal P_\mu S$ there is a model of $T \cup S_0$ of relation type $(M,R)$, then there is also a model of $T \cup S$ of relation type $(M,R)$.
	\end{enumerate}
\end{definition}

\begin{lemma}\label{lem:square-relation-type}
	Let $\cL$ be a logic, $\mu$ a regular cardinal, and $(M,R)$ a relation. Then $\cL$ is $[\mu, \mu]$-compact for relation type $(M,R)$ iff $\cL$ is $\mu$-chain compact for relation type $(M,R)$.
	\begin{proof}
		The same argument as that for Lemma \ref{lem:square-chain} works by letting all models considered have relation type $(M, R)$.
	\end{proof}
\end{lemma}

Note that the following characterisation uses first-order logic, from which we derive that the existence of a Berkeley cardinal is equivalent to a compactness property of first-order logic.
\begin{theorem}\label{thm:proto-Berkeley}
	The following are equivalent for a cardinal $\lambda$:
	\begin{enumerate}
		\item[(1)] $\lambda$ is $0$-proto Berkeley.
		\item[(2)] For every relation $(M,R)$ there is a cardinal $\mu < \lambda$ such that first-order logic is $\mu$-chain compact for relation type $(M,R)$.
	\end{enumerate}
	\begin{proof}
		Assume (2). Let $M$ be transitive with $\lambda \in M$. We seek to find an elementary $j: M \to M$ with $\crit(j) < \lambda$. For the relation $(M, \in)$, take a $\mu < \lambda$ according to (2). Consider:
		\[
		T = \text{ElDiag}(M,\in) \cup \{c_i \in c \wedge c < c_\mu \colon i < \mu \},
		\]
		which can be written as $T = \bigcup_{\beta < \mu} T_\beta$ by letting for $\beta < \mu$:
		\[
		T_\beta = 	\text{ElDiag}(M,\in) \cup \{c_i \in c \wedge c < c_\mu \colon i < \beta \}.
		\]
		Then $T_\beta$ has a model of relation type $(M,\in)$, namely $(M,\in, c^M, c_x^M)_{x \in M}$ where $c^M = \beta$. Thus, by (2), $T$ has a model $N = (N,R, c^N, c_x^N)_{x \in M}$ of relation type $(M,\in)$, i.e., there is an isomorphism $\pi:(N,R) \to (M,\in)$. Because $N \models T$, we get an elementary $j: (M,\in) \to (N,R)$ by $x \mapsto c_x^M$. Then $k = \pi \circ j: (M, \in) \to (M, \in)$ is elementary. By $T$, $c_\mu^N = j(\mu)$ has order type $> \mu$, and because $\pi$ is an isomorphism thus $k(\mu) = \pi(j(\mu))$ has order type $> \mu$. Thus, $\crit(k) \leq \mu < \lambda$. 
		
		\medskip
		
		Assume (1). Let $(M,R)$ be a relation and suppose for contradiction that for all $\mu < \lambda$ there is a first-order theory $T$ which can be written as an increasing union $T = \bigcup_{\beta < \mu} T_\beta$ such that every $T_\beta$ has a model of relation type $(M,R)$ but $T$ has no such model. Take $\eta > \lambda$ such that $V_\eta$ contains $(M,R)$ and, for every $\mu < \lambda$, a counterexample theory  $T$ as above alongside corresponding models of the $T_\beta$ for $\beta < \mu$ of relation type $(M,R)$. By $0$-proto Berkeleyness and Fact \ref{fact:proto-fix} there is $j: V_\eta \to V_\eta$ with $\crit(j) < \lambda$ and $j((M,R)) = (M,R)$. Fix $\mu = \crit(j)$ and, in $V_\eta$, consider a counterexample $T = \bigcup_{\beta < \mu} T_\beta$ corresponding to $\mu$, where the $T_\beta$ have models of relation type $(M,R)$, but $T$ has no such model. By elementarity, in $V_\eta$ and hence in $V$ we get that $j(T) = \bigcup_{\beta < j(\mu)} T_\beta^*$ for some theories $T_\beta^*$ where each $T_\beta^*$ has a model of relation type $j((M,R)) = (M,R)$. Since $\mu < j(\mu)$, then $j``T \subseteq T_\mu^*$ has such a model, which can be renamed to a model of $T$ of relation type $(M,R)$. Contradiction.		
	\end{proof}
\end{theorem}

\begin{corollary}\label{cor:existence-Berkeley}
	The following are equivalent:
	\begin{enumerate}
		\item[(1)] There is a Berkeley cardinal.
		\item[(2)] There is a cardinal $\lambda$ such that for every relation $(M,R)$ there is some $\mu < \lambda$ such that first-order logic is $\mu$-chain compact for relation type $(M,R)$.
	\end{enumerate}
	\begin{proof}
		If $\lambda$ is a Berkeley cardinal, it is also $0$-proto Berkeley and then satisfies (2) by Theorem \ref{thm:proto-Berkeley}. If $\lambda$ witnesses (2), then by Theorem \ref{thm:proto-Berkeley} it is $0$-proto Berkeley. Now the least $0$-proto Berkeley cardinal is Berkeley by Fact \ref{fact:proto-Berkeley}. 
	\end{proof}
\end{corollary}

To characterise Berkeleyness, we can use first-order infinitary logic, as well as stronger logics. 
\begin{theorem}\label{thm:Berkeley}
	The following are equivalent for a cardinal $\lambda$.
	\begin{enumerate}
		\item[(1)] $\lambda$ is Berkeley.
		\item[(2)] For every $\alpha < \lambda$ and every relation $(M,R)$ there is a cardinal $\alpha < \mu < \lambda$ such that $\cL_{\mu \omega}$ is $\mu$-chain compact for relation type $(M,R)$.
		\item[(3)] For every $\alpha < \lambda$ and every relation $(M,R)$ there is a cardinal $\alpha < \mu < \lambda$ such that $\cL^2_{\mu \omega}$ is $\mu$-chain compact for relation type $(M,R)$.
		\item[(4)] For every $\alpha < \lambda$ and every relation $(M,R)$ there is a cardinal $\alpha < \mu < \lambda$ such that $\sln_{\mu \omega}$ is $\mu$-chain compact for relation type $(M,R)$.
	\end{enumerate}
	\begin{proof}
		That (1) implies (2)--(4) can be seen similar to the proof (1) to (2) of Theorem \ref{thm:proto-Berkeley}, using $\alpha$-proto Berkeleyness. To accommodate for the stronger logics, simply chose a $V_\eta$ which computes their satisfaction relation correctly.
		
		That any of (2)--(4) imply (1) is similar to (2) implies (1) of the same theorem, using the fact that $\cL_{\mu \omega}$ and the stronger logics can define all ordinals $< \mu$ to push the critical point beyond $\alpha$. 
	\end{proof}
\end{theorem}

\subsection{Club Berkeley cardinals}
A cardinal $\lambda$ is \emph{club Berkeley} if $\lambda$ is regular and for all clubs $C \subseteq \lambda$ and all transitive $M$ with $\lambda \in M$ there is an elementary $j: M \to M$ with $\crit(j) \in C$ (cf. \cite{bagaria2019large}).

We use the following fact.
\begin{fact}[{Bagaria, Koellner, Woodin \cite[Lemma 3.3]{bagaria2019large}}]\label{fact:definable}
	For any set $a$ there is a transitive set $M$ such that $a \in M$ and $a$ is definable in $M$ without parameters.
\end{fact}

\begin{theorem}\label{thm:club-Berkeley}
	The following are equivalent for a regular cardinal $\lambda$:
	\begin{enumerate}
		\item[(1)] $\lambda$ is club Berkeley.
		\item[(2)] For all clubs $C \subseteq \lambda$ and all relations $(M,R)$ there is $\mu \in C$ such that $\cL_{\mu \omega}$ is $\mu$-chain compact for relation type $(M,R)$.
		\item[(3)] For all clubs $C \subseteq \lambda$ and all relations $(M,R)$ there is $\mu \in C$ such that $\cL^2_{\mu \omega}$ is $\mu$-chain compact for relation type $(M,R)$.
	\end{enumerate}
	\begin{proof}
		Clearly (3) implies (2). Assume (2), let $C \subseteq \lambda$ be club, and $M$ transitive with $\lambda \in M$ to show (1). Take a $\mu \in C$ according to (2) and the relation $(M,\in)$. The same argument as for Theorem \ref{thm:Berkeley} lets us derive an elementary $j: M \to M$ with $\crit(j) = \mu$. Because $\mu \in C$, we are done.
		
		\medskip
		
		Assume (1) and let $C \subseteq \lambda$ be club to show (3). Towards a contradiction suppose that there is a relation $(M,R)$ such that for all $\mu \in C$ there is an increasing $\cL_{\mu \omega}$-theory $T = \bigcup_{\beta < \mu} T_\beta$ such that every $T_\beta$ has a model of relation type $(M,R)$ but $T$ has no such model. Let $\eta > \mu$ be large enough to contain such a counterexample theory alongside models of relation type $(M,R)$ of the filtrating theories for every $\mu \in C$. By Fact \ref{fact:definable}, let $N$ be transitive such that $((M,R), V_\eta) \in N$ and is definable in $N$ without parameters. Then $V_\eta \subseteq N$. In particular, $\lambda \in M$. By (1), let $j: N \to N$ with $\mu := \crit(j) \in C$. Because $((M,R), V_\eta)$ is definable in $N$, $j$ fixes $(M,R)$ and $V_\eta$. Because $V_\eta$ contains counterexample theories alongside the corresponding models of the filtrating theories, $N$ also contains the counterexample $T = \bigcup_{\beta < \mu} T_\beta$ corresponding to $\mu$ along with models of relation type $(M,R)$ of the $T_\beta$ in $V_\eta$. Note that $N$ is correct about $\cL^2_{\mu \omega}$-satisfaction for members of $V_\eta$, because $V_\eta \subseteq N$ and so $N$ contains the full power set of these models. So $N$ understands this situation. By elementarity, thus $N$ believes that $j(T)$ is filtrated by an increasing sequence of theories $(T_\beta^* \colon \beta < j(\mu))$ and each $T_\beta^*$ has a model of relation type $j((M,R)) = (M,R)$ in $j(V_\eta) = V_\eta$. Then as $\mu < j(\mu)$, $T_\mu^* \supseteq j``T$ has a model $A$ of relation type $(M,R)$ in $V_\eta \subseteq N$. Again, $N$ is correct about $A \in V_\eta \subseteq N$ satisfying the $\cL^2_{\mu \omega}$-theory $j``T$. Thus, $A$ can be renamed to model of $T$ in $V$. Contradiction.
	\end{proof}
\end{theorem}

\section{Exacting cardinals}\label{sec:exacting}

A cardinal $\lambda$ is \emph{exacting} if for every $\zeta > \lambda$ there is an elementary substructure $X \prec V_\zeta$ such that $V_\lambda \cup \{ \lambda \} \subseteq X$ and an elementary embedding $j: X \to V_\zeta$ with $j(\lambda) = \lambda$ and $\crit(j) < \lambda$ (cf. \cite{aguilera2024large}).

Note that an embedding witnessing exactingness is similar to one witnessing rank-Berkeleyness, just that the domain of the embedding is shrunk to some elementary substructure of $V_\zeta$. In this sense, exacting cardinals can be conisdered to be an AC compatible version of rank-Berkeley cardinals. There are many equivalent characterisations of exacting cardinals. We will make use of the following two.
\begin{fact}[Aguilera, Bagaria, Goldberg, Lücke \cite{aguilera2025large}]\label{fact:exacting}
	The following are equivalent for a cardinal $\lambda$:
	\begin{enumerate}
		\item[(1)] $\lambda$ is exacting.
		\item[(2)] For every $\alpha < \lambda < \zeta$ there is an elementary substructure $X \prec V_\zeta$ such that $V_\lambda \cup \{ \lambda \} \subseteq X$ and an elementary embedding $j: X \to V_\zeta$ with $j(\lambda) = \lambda$ and $\alpha < \crit(j) < \lambda$. 
		\item[(3)] For $\zeta$ the smallest member of $C^{(2)}$ above $\lambda$ there is an elementary substructure $X \prec V_\zeta$ such that $V_\lambda \cup \{ \lambda \} \subseteq X$ and an elementary embedding $j: X \to V_\zeta$ with $j(\lambda) = \lambda$ and $\crit(j) < \lambda$.
	\end{enumerate}
\end{fact}

We give yet another characterisation of exactingness, which shows that exactingness gives elementary embeddings as above for \emph{every} substructure of $V_\zeta$ for $\zeta > \lambda$. This makes the notion more directly applicable to derive model-theoretic consequences.

\begin{lemma}\footnote{This result is joint with Hope Duncan and Asaf Karagila. I want to thank Hope and Asaf for conversations on the axiom of choice and on exacting cardinals during which we proved this characterisation. An earlier version of this article showed that the compactness properties of Theorem \ref{thm:exacting-compactness} are equivalent to (2) of this lemma, but was missing the equivalence of (2) to exactingness. Instead, it contained a proof that (2) follows from $\lambda$ being \emph{ultraexacting} (cf. \cite{aguilera2024large} for the definition of ultraexactingness).}\label{lem:multiexacting}
	Let $\lambda$ be a cardinal. The following are equivalent:
	\begin{enumerate}
		\item[(1)] $\lambda$ is exacting.
		\item[(2)] For every $\alpha < \lambda < \zeta$ and every elementary substructure $X \prec V_\zeta$ with $|X| = \lambda$ and $V_\lambda \cup \{\lambda\} \subseteq V_\zeta$ there is an elementary embedding $j: X \to V_\zeta$ with $\alpha < \crit(j) < \lambda$ and $j(\lambda) = \lambda$.
		\end{enumerate}
		\begin{proof}
			Clearly (2) implies that $\lambda$ is exacting. So it is sufficient to show that (1) implies (2). Assume (1), let $\alpha < \lambda$, and suppose towards a contradiction that $\zeta > \lambda$ is smallest such that there is an elementary substructure $X \prec V_\zeta$ of size $|X|=\lambda$ containing $V_\lambda \cup \{\lambda\}$ but no elementary $k: X \to V_\zeta$ with $\alpha < \crit(k) < \lambda$ and $k(\lambda) = \lambda$. Let $\eta > \zeta$ be such that $V_\eta$ is sufficiently correct to reflect the definition of $\zeta$. By exactingness and Fact \ref{fact:exacting}, there is an elementary substructure $Y \prec V_\eta$ with $V_\lambda \cup \{\lambda\} \subseteq Y$ and an elementary embedding $j: Y \to V_\eta$ such that $\alpha < \crit(j) < \lambda$ and $j(\lambda) = \lambda$. Note that since $\zeta$ is definable from $\lambda$, we have $\zeta, V_\zeta \in Y$. Because $V_\eta$ beliefs there is a substructure of $V_\zeta$ which constitutes counterexample to (2), by elementarity of $Y$ there is such a counterexample $X \in Y$, i.e., $Y$ satisfies:
			\begin{enumerate}
				\item[($\ast$)] $X \prec V_\zeta$, $|X| = \lambda$, $V_\lambda \cup \{\lambda\} \subseteq X$, and there is no elementary embedding $k: X \to V_\zeta$ with $\alpha < \crit(k) < \lambda$ and $k(\lambda) = \lambda$. 
			\end{enumerate}
			Again by elementarity, this all holds in $V_\eta$ and thus in $V$. Now note that because $|X| = \lambda$, $Y$ has a surjection $f: \lambda \to X$. Since $\lambda \subseteq Y$, it follows that $X = f``\lambda \subseteq Y$. In particular, we can consider the restriction $k = j \upharpoonright X$. Note that by definability of $\zeta$ from $\lambda$ and $\alpha$, and because $j(\lambda) = \lambda$ and $j(\alpha) = \alpha$, we get $j(\zeta) = \zeta$. Thus, $k: X \to V_\zeta$. Clearly, $\alpha < \crit(j) = \crit(k) < \lambda$ and $k(\lambda) = j(\lambda) = \lambda$. By elementarity of $j$, it follows that $V_\eta$ believes that $j(X)$ is an elementary substructure of $V_{j(\zeta)} = V_\zeta$. Hence, the following chain of equivalences establishes that $k$ is elementary.
			\[
			\begin{split}
				X \models \varphi(a) & \text{ iff }  V_\eta \models ``X \models \varphi(a)" \\
				& \text{ iff } Y \models ``X \models \varphi(a)" \\
				& \text{ iff } V_\eta \models ``j(X) \models \varphi(j(a))" \\
				& \text{ iff } V_\eta \models ``V_\zeta \models \varphi(j(a))" \\
				& \text{ iff } V_\zeta \models \varphi(j(a)) \\
				& \text{ iff } V_\zeta \models \varphi(k(a)).
			\end{split}
			\]
			But then the existence of $k$ contradicts that $(\ast)$ holds in $V$. 
		\end{proof}
\end{lemma}

\subsection{Compactness for rank types with theory size restriction}

The following variants to $\mu$-chain compactness and $[\kappa,\mu]$-compactness for type $(\alpha,\beta)$ characterise exactingness. Note that the only difference to the earlier notions is the added restriction of the size of $T$. 

	\begin{definition}
	Fix a logic $\cL$, cardinals $\kappa, \mu$, and $\lambda$, and ordinals $\alpha$ and $\beta$.
	\begin{enumerate}
		\item[(1)] $\cL$ is \emph{$(\mu,\lambda)$-chain compact for rank type $(\alpha, \beta)$} if for every theory $T \subseteq \cL$ with $|T| = \lambda$ which can be written as an increasing union $T = \bigcup_{\alpha < \mu} T_\alpha$, if every $T_\alpha$ has a model of rank type $(\alpha, \beta)$, then also $T$ has a model of rank type $(\alpha, \beta)$.
		\item[(2)] $\cL$ is \emph{$[\kappa, \mu, \lambda]$-compact for rank type $(\alpha, \beta)$} if for all theories $T,S \subseteq \cL$ with $|T| = \lambda$ and $|S| = \kappa$, if for every $S_0 \in \mathcal P_\mu S$ there is a model of $T \cup S_0$ of rank type $(\alpha, \beta)$, then there is also a model of $T \cup S$ of rank type $(\alpha, \beta)$.
	\end{enumerate}
\end{definition}

The equivalence between $\mu$-chain compactness and $[\mu,\mu]$-compactness carries over:
\begin{lemma}\label{lem:square-chain-type-size}
	Let $\cL$ be a logic and $\mu$ a regular cardinal. Then $\cL$ is $[\mu, \mu,\lambda]$-compact for rank type $(\alpha, \beta)$ iff $\cL$ is $(\mu,\lambda)$-chain compact for rank type $(\alpha, \beta)$.
	\begin{proof}
		The same argument as that for Lemma \ref{lem:square-chain} works simply by restricting the size of the theories and rank types of the models appropriately.
	\end{proof}
\end{lemma}

\begin{theorem}\label{thm:exacting-compactness}
	The following are equivalent for $\lambda = \beth_\lambda$:
	\begin{enumerate}
		\item[(1)] $\lambda$ is exacting. 
		\item[(2)] For $\zeta > \lambda$ the smallest member of $C^{(2)}$ above $\lambda$ there is a cardinal $\mu < \lambda$ such that $\cL^2$ is $(\mu,\lambda)$-chain compact for rank type $(\zeta, \lambda)$. 
		\item[(3)] For every $\zeta > \lambda$ there is a cardinal $\mu < \lambda$ such that $\cL^2$ is $(\mu,\lambda)$-chain compact for rank type $(\zeta, \lambda)$. 
		\item[(4)] For every $\alpha < \lambda < \zeta$ there is a cardinal $\alpha < \mu < \lambda$ such that $\cL^2_{\mu \omega}$ is $(\mu,\lambda)$-chain compact for rank type $(\zeta, \lambda)$. 
		\item[(5)] For every natural number $n$ and every $\zeta > \lambda$ there is a cardinal $\mu < \lambda$ such that $\sln$ is $(\mu,\lambda)$-chain compact for rank type $(\zeta, \lambda)$. 
		\item[(6)] For every natural number $n$ and every $\alpha < \lambda < \zeta$ there is a cardinal $\alpha < \mu < \lambda$ such that $\sln_{\mu \omega}$ is $(\mu,\lambda)$-chain compact for rank type $(\zeta, \lambda)$. 
	\end{enumerate}
	\begin{proof}
		Clearly, (2) is implied by all of (3)--(6). Moreover, (6) implies all of (2)--(5). So it is sufficient to show that (1) implies (6) and that (2) implies (1).
		
		\medskip
		
		For (1) implies (6), let $n$ be any natural number and $\alpha < \lambda$. Towards a contradiction, suppose that there is a smallest $\zeta > \lambda$ such that for all $\mu$ with $\alpha < \mu < \lambda$ there is a theory $T = \bigcup_{\beta < \mu} T_\beta \subseteq \cL^2_{\mu \omega}$ of size $|T| = \lambda$ such that all $T_\beta$ have models of type $(\zeta, \lambda)$ but $T$ has no such model. For every $\alpha < \mu < \lambda$ fix one such theory $T_\mu$ of size $|T_\mu| = \lambda$ with its corresponding decomposition $T_\mu = \bigcup_{\beta < \mu} T_{\beta,\mu}$. The proof proceeds similar to those of the characterisations of rank-Berkeleyness and Berkeleyness, using that we can put all these theories into one elementary substructure of an appropriate $V_\eta$. Fix an $\eta > \zeta$ such that $V_\eta$ reflects the definition of $\zeta$ from $\alpha$ and $\lambda$, is correct about $\sln_{\mu \omega}$-satisfaction for $\mu < \lambda$, and contains for all $\mu < \lambda$ all the $T_\mu$ and models $A_{\beta,\mu}$ of the $T_{\beta,\mu}$ of type $(\zeta, \lambda)$. Let $X\prec V_\eta$ with $V_\lambda \cup \{\lambda\} \cup \bigcup_{\alpha < \mu < \lambda} T_\mu \cup \{T_\mu \colon \alpha < \mu < \lambda\} \cup \{T_{\beta,\mu} \colon \alpha < \mu < \lambda, \beta < \mu\} \cup \{A_{\beta,\mu} \colon \alpha < \mu < \lambda, \beta < \mu\} \subseteq X$ of size $|X| = \lambda$. Note that we can chose $X$ of size $\lambda$ because all the $T_\mu$'s have size $\lambda$. Because $\lambda$ is exacting, by Lemma \ref{lem:multiexacting} there is an elementary $j: X \to V_\eta$ such that $j(\lambda) = \lambda$ and $\alpha < \crit(j) < \lambda$. Definability of $\zeta$ from $\alpha$ and $\lambda$, correctness of $V_\eta$, and elementarity of $X$ as a substructure of $V_\eta$ together imply $\zeta \in X$. Then, again by correctness of $\eta$ and definability of $\zeta$ from $\alpha$ and $\lambda$, and because $j(\alpha) = \alpha$ and $j(\lambda) = \lambda$, it follows that also $j(\zeta) = \zeta$. Now fix $\mu = \crit(j)$. Then $V_\eta$ verifies that $T_\mu$ is the increasing union of the $T_{\beta,\mu}$, each $T_{\beta,\mu}$ has a model of type $(\zeta,\lambda)$, but $T_\mu$ has no such model. By elementarity of $X$ as a substructure, the same holds in $X$. By elementarity of $j$, for some sequence of theories $(T_\beta^* \colon \beta < j(\mu))$, we get that $V_\eta$ satisfies: 
		\begin{quote}
			$j(T_\mu) = \bigcup_{\beta < j(\mu)} T_\beta^*$ is an increasing union and every $T_\beta^*$ has a model of type $(j(\zeta), j(\lambda)) = (\zeta, \lambda)$, but $j(T_\mu)$ has no such model.
		\end{quote}
		Note that $j``T_\mu \subseteq T_\mu^*$. Since $\mu < j(\mu)$, thus $V_\eta$ has a model $A$ of $j``T_\mu$ of type $(\zeta, \lambda)$. Because $\crit(j) = \mu$, $j``T_\mu$ is really an $\cL^2_{\mu \omega}$-theory. Because $V_\eta$ is correct about $\cL^2_{\mu \omega}$ satisfaction and types of models, $A$ is really a model of $j``T_\mu$ of rank type $(\zeta, \lambda)$ and thus can be renamed to a model of $T_\mu$ of rank type $(\zeta,\lambda)$. But we assumed that there is no such model. Contradiction.
		
		\medskip
		
		For (2) implies (1), let $\zeta > \lambda$ be the smallest member of $C^{(2)}$ above $\lambda$, and consider any $X \prec V_\zeta$ of size $|X| = \lambda$ with $V_\lambda \cup \{\lambda\}\subseteq X$. By Fact \ref{fact:exacting}, it is sufficient to find an elementary $j: X \to V_\zeta$ with $j(\lambda) = \lambda$ and some critical point $\crit(j) < \lambda$. Take $\mu < \lambda$ according to (3). Consider the theory:
		\[\begin{split}
			T =  &\text{ ElDiag}(X, \in) \cup \{\Phi \} \cup \{c_i < c \wedge c < c_\mu \colon i < \mu \} \\ & \cup \{\forall x(P(x) \leftrightarrow \rk(x)<c_\lambda) \}.
			\end{split}
		\]
		Note that because $|X| = \lambda$, also $|\text{ElDiag}(X)| = \lambda$, and thus $|T| = \lambda$. As usual, we can write $T = \bigcup_{\beta < \mu} T_\beta$ considering those bits of $T$ that contain $``c_i < c \wedge c < c_\mu"$ for $i < \beta$. Then $T_\beta$ is satisfiable by a model of rank type $(\zeta, \lambda)$, namely by $A = (V_\zeta, \in, P^A, d^A, c_x^A)_{x \in V_\zeta}$ letting $P^A = V_\lambda$, $c^A = \beta$, and $c_x^A = x$. Note that this satisfies $T \supseteq \text{ElDiag}(X)$ because $X \prec V_\zeta$. Thus, by (2), $T$ has a model $B$ of rank type $(\zeta, \lambda)$. We have $j: X \to B$ elementary with $x \mapsto c_x^B$. Because $B \models \Phi$, without loss of generality $B = V_\eta$ for some $\eta$. The usual argument gives $j(\lambda) = \lambda$ and $\crit(j) \leq \mu < \lambda$.
	\end{proof}
\end{theorem}

We derive yet another characterisation of exactingness of $\lambda$ from this theorem, which shows for fixed $\zeta> \lambda$, all substructures of $V_\zeta$ admit elementary embeddings with the same critical point. Moreover, the uniform critical point can be chosen arbitrarily high below $\lambda$.
\begin{corollary}\label{cor:char-exacting}
	A cardinal $\lambda$ is exacting iff for every $\alpha < \lambda < \zeta$ there is $\alpha < \mu < \lambda$ such that for every $X \prec V_\zeta$ with $|X| = \lambda$ and $V_\lambda \cup \{\lambda\} \subseteq X$ there is $j: X \to V_\zeta$ with $j(\lambda) = \lambda$ and $\crit(j) = \mu$.
	\begin{proof}
		For the substantial direction, let $\alpha < \lambda < \zeta$ and take a $\mu$ according to (4) of Theorem \ref{thm:exacting-compactness}. We claim that $\mu$ is as desired. So let $X \prec V_\zeta$ be any elementary substructure with $|X| = \lambda$ and $V_\lambda \cup \{\lambda\} \subseteq X$. Consider:
		\[
		\begin{split}
			T = & \text{ ElDiag}(X, \in) \cup \{\Phi \} \cup \{c_i \in d \wedge |d| < c_\mu \colon i < \mu \} \\ & \cup \{``P \text{ contains exactly the sets of rank } < c_\lambda" \} \cup \{\sigma_i(c_i) \colon i < \mu \}.
		\end{split}
		\]
		The theory is as the one considered in the proof of the above theorem, just the last part is added in which we use $\cL_{\mu \omega}$ to define every ordinal below $\mu$. Using (4), the same argument as before gives us an elementary embedding $j: X \to V_\zeta$ with $\crit(j) \leq \mu$. Because we included definitions of the ordinals below $\mu$, we now get that $\crit(j) = \mu$.
	\end{proof}
\end{corollary}

If for $\zeta > \lambda$, the cardinal $\mu < \lambda$ is as in the characterisation above (i.e., for every $X \prec V_\zeta$ with $|X| = \lambda$ and $V_\lambda \cup \{\lambda\} \subseteq X$ there is $j: X \to V_\zeta$ with $j(\lambda) = \lambda$ and $\crit(j) = \mu$), let us say that $\mu$ is \emph{critical for $\lambda$'s exactingness at $\zeta$}. Exacting cardinals come with $\mu$ which are critical for $\lambda$'s exactingness at many $\zeta$.

\begin{corollary}\label{cor:char-exacting2}
	A cardinal $\lambda$ is exacting iff for every $\alpha < \lambda$ there is $\alpha < \mu < \lambda$ such that the class of ordinals $\zeta$ such that $\mu$ is critical for $\lambda$'s exactingness at $\zeta$ is unbounded.
	\begin{proof}
		This follows directly from the previous corollary as there is a proper class of ordinals $\zeta > \lambda$ but only $\lambda$ many candidates for cardinals being critical for $\lambda$'s exactingness at $\zeta$.
	\end{proof}
\end{corollary}
We are operating close to inconsistency here, as the considered class cannot be closed.
\begin{proposition}\label{prop:exacting-inconsistent}
	There are no $\mu < \lambda$ such that for a club class $C$ of ordinals, $\mu$ is critical for $\lambda$'s exactingness at every $\zeta \in C$. In particular, there are no $\mu < \lambda$ such that for every $\zeta > \lambda$ and every $X \prec V_\zeta$ with $|X| = \lambda$ and $V_\lambda \cup \{\lambda\} \subseteq X$ there is an elementary embedding $j: X \to V_\zeta$ with $j(\lambda) = \lambda$ and $\crit(j) = \mu$.
	\begin{proof}
		Suppose that $\mu$ is smallest such that the class $C$ of ordinals $\zeta$ such that $\mu$ is critical for $\lambda$'s exactingness at $\zeta$ is club. Because $C$ is club, by the reflection theorem there is some $\eta \in C$ such that $V_\eta$ is correct about this definition and so, in $V_\eta$, $\mu$ is definable as the smallest cardinal with the above property. Take any elementary substructure $X \prec V_\eta$ of size $|X| = \lambda$ and with $V_\lambda \cup \{\lambda\} \subseteq X$. Because $\eta \in C$ there is some elementary embedding $j: X \to V_\eta$ with $j(\lambda) = \lambda$ and $\crit(j) = \mu$. Further, by elementarity of $X$ as a substructure, $\mu$ is also definable from $\lambda$ in $X$. But then, since $j(\lambda) = \lambda$, it follows that $j(\mu) = \mu$. This contradicts that $\mu = \crit(j)$.
	\end{proof}
\end{proposition}

The following is reminiscent of Makowsky's result on VP (cf. Fact \ref{fact:makowsky}).
\begin{corollary}\label{cor:class-exacting}
	There are unboundedly many exacting cardinals if and only if for every logic $\cL$ there is a cardinal $\lambda = \beth_\lambda$ such that for every $\alpha < \lambda < \zeta$ there is $\alpha < \mu < \lambda$ such that $\cL$ is $(\mu,\lambda)$-chain compact for rank type $(\zeta, \lambda)$. 
	\begin{proof}
		If the assumption holds and $\cL$ is some logic, take some natural number $n$ and a cardinal $\kappa$ such that $\cL \leq \sln_{\kappa \omega}$. Take an exacting cardinal $\lambda > \kappa$. We show that $\lambda$ witnesses the conclusion. Let $\alpha < \lambda$. By Theorem \ref{thm:exacting-compactness} there is $\mu$ with $\max({\kappa,\alpha}) < \mu < \lambda$ such that $\sln_{\mu \omega}$ is $(\mu,\lambda)$-chain compact for rank type $(\zeta,\lambda)$. Then also $\cL \leq \sln_{\mu \omega}$ is $(\mu,\lambda)$-chain compact for rank type $(\zeta,\lambda)$. 
		
		\medskip
		
		For the other direction, if $\lambda$ witnesses the conclusion for $\cL = \cL^2_{\kappa \omega}$, then $\lambda$ is exacting because it witnesses Theorem \ref{thm:exacting-compactness}.(2). So it is sufficient that then also $\lambda \geq \kappa$. So suppose $\lambda < \kappa$. For $\zeta > \kappa$, take $\mu < \lambda$ such that $\cL$ is $(\mu,\lambda)$-chain compact for rank type $(\zeta, \lambda)$. Take an elementary substructure $X \prec V_\zeta$ of size $|X| = \lambda$ and $V_\lambda \cup \{\lambda\} \subseteq X$. Note that we have the formula $\sigma_\mu(x)$ defining $\mu$ available because $\mu < \lambda < \kappa$. The theory
		\[\begin{split}
			T =  &\text{ ElDiag}(X, \in) \cup \{\Phi \} \cup \{c_i < c \wedge c < c_\mu \colon i < \mu \} \\ & \cup \{\forall x(P(x) \leftrightarrow \rk(x)<c_\lambda) \} \cup \{\sigma_\mu(c_\mu)\}
		\end{split}
		\]
		has no model, because it forces $c_\mu$ to both have and not have order type $\mu$. But if $T_\beta$ for $\beta < \mu$ as usual contains those bits of $T$ which have the sentences $``c_i < c \wedge c < c_\mu"$ for $i < \beta$, then $T_\beta$ has a mode $A$ of rank type $(\zeta ,\lambda)$ by expanding $V_\zeta$ by $c^A = \beta$ and $P^A = V_\lambda$. This contradicts that $\cL$ is $(\mu,\lambda)$-chain compact for rank type $(\zeta, \lambda)$. 
	\end{proof}
\end{corollary}

\subsection{Compactness for size types}
Instead of rank types, we may also consider \emph{size types}, where for cardinals $\zeta \geq \lambda$ and a vocabulary $\tau$ with a designated unary relation symbol $P \in \tau$, a $\tau$-structure $A$ has \emph{size type} $(\zeta, \lambda)$ if $|A| = \zeta$ and $|P^A| = \lambda$. Then we can define the obvious analogues of the properties considered above, \emph{$(\mu,\lambda)$-chain compactness for size type $(\zeta,\lambda)$} and \emph{$[\mu,\mu,\lambda]$-compactness for size type $(\zeta,\lambda)$}. In Theorem \ref{thm:exacting-compactness}, we can equivalently substitute consideration of size types instead of rank types. Note that being of size type $(\zeta,\lambda)$ is closed under isomorphism, while being of rank type $(\zeta,\lambda)$ is not. As in model theory we are usually considering properties invariant under isomorphism, the characterisation using size types is arguably the model-theoretically cleaner one. We decided to present the version with rank types instead to underline the analogy to rank-Berkeley cardinals. Note that in the choiceless setting, the version with size types does not work, as, e.g., it might be that some $V_\zeta$ does not have a cardinality. 

\section{Topological compactness}\label{sec:topology}
The notion of \emph{$[\kappa, \mu]$-compactness} was first studied in topology by Alexandroff and Urysohn. We thus state our results in terms of compactness principles for topological spaces. In particular, the results show how the topological property leads to violations of the axiom of choice and of $V=\HOD$ if applied to the right spaces. The following is the notion studied in topology.
\begin{definition}
	Let $X$ be a topological space, $\kappa$ and $\mu$ cardinals. Then $X$ is called \emph{$[\kappa, \mu]$-compact} iff for all families $(X_i)_{i < \kappa}$ of closed sets with $\bigcap_{i < \kappa} X_i = \emptyset$ there is $B \in \mathcal P_\mu \kappa$ such that $\bigcap_{i \in B} X_i = \emptyset$. 
\end{definition}
Recall that compactness of a logic is connected to compactness of topological spaces. One (folklore) way of spelling this out is the following. For a logic $\cL$ and a vocabulary $\tau$, let $\equiv_{\cL[\tau]}$ be the equivalence relation on the class of $\tau$-structures where
\[
A \equiv_{\cL[\tau]} B \text{ iff for all } \varphi \in \cL[\tau](A \models_\cL \varphi \text{ iff } B \models_\cL \varphi).
\]
Write $[A]$ for the Scott equivalence class \[[A] = \{B \colon B \equiv_{\cL[\tau]} A \text{ and } \rk(B) \text{ is minimal}\}.\] Let $\Top(\cL[\tau])$ be the topological space on $X(\cL[\tau]) = \{[A] \colon A \text{ is a } \tau \text{-structure} \}$ with the topology which has as a subbasis the following sets $Y_\varphi$, where $\varphi \in \cL[\tau]$:
\[
Y_\varphi = \{[A] \in X(\cL[\tau]) \colon A \models_\cL \varphi \}.
\] 
It is a classic observation that the compactness theorem for first-order logic is equivalent to the statement that for every vocabulary $\tau$, the topological space $\Top(\cL_{\omega \omega}[\tau])$ is compact in the topological sense. Mannila first observed that a similar equivalence transfers to $[\kappa,\mu]$-compactness as introduced by Makowsky and Shelah and topological $[\kappa,\mu]$-compactness as introduced by Alexandroff and Urysohn.
\begin{fact}[Mannila \cite{mannila1983topological}]\label{fact:mannila}
	Let $\cL$ be a logic, $\kappa$ and $\mu$ cardinals. Then $\cL$ is $[\kappa,\mu]$-compact if and only if for every $\tau$, the space $\Top(\cL[\tau])$ is $[\kappa,\mu]$-compact.
\end{fact}
The proof easily carries over to our compactness variants for rank and relation types, but to make it work we need to consider model spaces which restrict attention to models of the correct types. 

For ordinals $\alpha$ and $\beta$, write $\Top^{(\alpha,\beta)}(\cL[\tau])$ for the topological space on the set \[X^{(\alpha, \beta)}(\cL[\tau]) = \{[A] \colon A \text{ is a } \tau \text{-structure of rank type } (\alpha,\beta) \}\] with the topology defined analogously to before.

For a relation $(M,R)$, write $\Top^{(M,R)}(\cL[\tau])$ for the topological space on the set \[X^{(M,R)}(\cL[\tau]) = \{[A] \colon A \text{ is a } \tau \text{-structure of relation type } (M,R) \}\] with the topology defined analogously to before. Then:
\begin{lemma}[ZF]\label{lem:topology-logic}
	Take a logic $\cL$, cardinals $\kappa,\mu$, ordinals $\alpha$ and $\beta$, and a relation $(M,R)$. Then:
	\begin{enumerate} 
	
	\item[(i)] $\cL$ is $[\kappa,\mu]$-compact for rank type $(\alpha,\beta)$ if and only if for all vocabularies $\tau$, the space $\Top^{(\alpha,\beta)}(\cL[\tau])$ is $[\kappa,\mu]$-compact.
	
	\item[(ii)] $\cL$ is $[\kappa,\mu]$-compact for relation type $(M,R)$ if and only if for all vocabularies $\tau$, the space $\Top^{(M,R)}(\cL[\tau])$ is $[\kappa,\mu]$-compact.
	\end{enumerate}
	\begin{proof}
		The proof in \cite[Theorem 1]{mannila1983topological} carries over restricting the rank types or relation types of considered models where appropriate.
	\end{proof}
\end{lemma}
\begin{lemma}\label{lem:topology-logic2}
		Take a logic $\cL$, cardinals $\kappa,\mu$, a cardinal $\lambda\geq\kappa$ such that for vocabularies $\tau$ with $|\tau| \leq \lambda$ also $|\cL[\tau]| \leq \lambda$, and ordinals $\alpha$ and $\beta$. Then $\cL$ is $[\kappa,\mu,\lambda]$-compact for rank type $(\alpha,\beta)$ if and only if for all vocabularies $\tau$ of size $|\tau| \leq \lambda$, the space $\Top^{(\alpha,\beta)}(\cL[\tau])$ is $[\kappa,\mu]$-compact. 
		\begin{proof}
			Again, the proof in \cite[Theorem 1]{mannila1983topological} carries over additionally restricting sizes of theories and vocabularies.
		\end{proof}
\end{lemma}

The following give examples of how the logical compactness characterisations transfer to topological results. All other equivalences proven in this article have similar topological reformulations. 
\begin{corollary}[ZF]\label{cor:rank-Berkeley-topology}
	The following are equivalent for a cardinal $\lambda$:
	\begin{enumerate}
		\item[(1)] $\lambda$ is rank-Berkeley.
		\item[(2)] For every $\alpha < \lambda < \zeta$ there is a cardinal $\alpha < \mu < \lambda$ such that for every vocabulary $\tau$ the space $\Top^{(\zeta,\lambda)}(\cL^2_{\mu \omega}[\tau])$ is $[\mu,\mu]$-compact. 
	\end{enumerate}
	\begin{proof}
		Follows from Theorem \ref{thm:rank-Berkeley} using the equivalence of topological $[\mu,\mu]$-compactness,  $[\mu,\mu]$-compactness for rank type $(\zeta,\lambda$), and $\mu$-chain compactness for rank type $(\zeta, \lambda)$ (Lemmas \ref{lem:square-chain-type} and \ref{lem:topology-logic}).
	\end{proof}
\end{corollary}
\begin{corollary}[ZF]\label{cor:Berkeley-topology}
	The following are equivalent for a cardinal $\lambda$:
	\begin{enumerate}
		\item[(1)] $\lambda$ is Berkeley.
		\item[(2)] For every relation $(M,R)$ there is a cardinal $\alpha < \mu < \lambda$ such that for every vocabulary $\tau$ the space $\Top^{(M,R)}(\cL_{\mu \omega}[\tau])$ is $[\mu,\mu]$-compact. 
	\end{enumerate}
	\begin{proof}
		Follows from Theorem \ref{thm:Berkeley} using the equivalence of topological $[\mu,\mu]$-compactness, $[\mu,\mu]$-compactness for relation type $(\zeta,\lambda$), and $\mu$-chain compactness for relation type $(\zeta, \lambda)$ (Lemmas \ref{lem:square-relation-type} and \ref{lem:topology-logic}).
	\end{proof}
\end{corollary}
\begin{corollary}[ZF]\label{cor:Berkeley-topology2}
	The following are equivalent:
	\begin{enumerate}
		\item[(1)] There is a Berkeley cardinal.
		\item[(2)] There is a cardinal $\lambda$ such that for every relation $(M,R)$ there is a cardinal $\mu < \lambda$ such that for every vocabulary $\tau$ the space $\Top^{(M,R)}(\cL_{\omega \omega}[\tau])$ is $ [\mu,\mu]$-compact.
	\end{enumerate}
	\begin{proof}
		Follows from Corollary \ref{cor:existence-Berkeley} using the equivalence of topological $[\mu,\mu]$-compactness, $[\mu,\mu]$-compactness for relation type $(\zeta,\lambda$), and $\mu$-chain compactness for relation type $(\zeta, \lambda)$ (Lemmas \ref{lem:square-relation-type} and \ref{lem:topology-logic}).
	\end{proof}
\end{corollary}

\begin{corollary}[ZF]\label{cor:club-Berkeley-topology}
	The following are equivalent for a regular cardinal $\lambda$:
	\begin{enumerate}
		\item[(1)] $\lambda$ is club Berkeley.
		\item[(2)] For all clubs $C \subseteq \lambda$ and all relations $(M,R)$ there is $\mu \in C$ such that for every vocabulary $\tau$ the space $\Top^{(M,R)}(\cL_{\mu \omega}[\tau])$ is $[\mu,\mu]$-compact. 
	\end{enumerate}
	\begin{proof}
		Follows from Theorem \ref{thm:club-Berkeley} using the equivalence of topological $[\mu,\mu]$-compactness, $[\mu,\mu]$-compactness for relation type $(\zeta,\lambda$), and $\mu$-chain compactness for relation type $(\zeta, \lambda)$ (Lemmas \ref{lem:square-relation-type} and \ref{lem:topology-logic}).
	\end{proof}
\end{corollary}

\begin{corollary}\label{cor:exacting-topology}
	The following are equivalent for a cardinal $\lambda = \beth_\lambda$:
	\begin{enumerate}
		\item[(1)] $\lambda$ is exacting.
		\item[(2)] For every $\alpha < \lambda < \zeta$ there is a cardinal $\alpha < \mu < \lambda$ such that for every vocabulary $\tau$ of size $|\tau| \leq \lambda$ the space $\Top^{(\zeta,\lambda)}(\cL^2_{\mu \omega}[\tau])$ is $[\mu,\mu]$-compact.
	\end{enumerate}
	\begin{proof}
		Follows from Theorem \ref{thm:exacting-compactness} using the equivalence of topological $[\mu,\mu]$-compactness,  $[\mu,\mu,\lambda]$-compactness for rank type $(\zeta,\lambda$), and $(\mu,\lambda)$-chain compactness for rank type $(\zeta, \lambda)$ (Lemmas \ref{lem:square-chain-type-size} and \ref{lem:topology-logic2}).
	\end{proof}
\end{corollary}

\section{Open questions}
The characterisation of exacting cardinals in terms of compactness derives from the one of rank-Berkeleyness by restricting the theory size. We can consider an analogous weakening of the characterisation of Berkeleyness.
\begin{definition}
	Let $\lambda$ and $\mu$ be cardinals, $(M,R)$ a relation. Say that a logic $\cL$ is \emph{$(\mu,\lambda)$-chain compact for relation type $(M,R)$} if for every theory $T \subseteq \cL$ with $|T| = \lambda$ and which can be written as an increasing union $T = \bigcup_{\beta < \mu} T_\beta$ where each $T_\beta$ has a model of relation type $(M,R)$ there is a model of $T$ of relation type $(M,R)$.
\end{definition}
\begin{question}
	 What is the large cardinal strength of the existence of a cardinal $\lambda$ such that for every $\alpha < \lambda$ and every relation $(M,R)$ there is a cardinal $\alpha < \mu < \lambda$ such that $\cL_{\mu \omega}$ (or some other logic) is $\mu$-chain compact for relation type $(M,R)$? Is such a cardinal consistent with choice?
\end{question}
If the answer to the last question is positive, such a cardinal might be a Berkeley-analogue consistent with choice, similar to how exacting cardinals are an analogue of rank-Berkeleyness consistent with choice.
\begin{question}
	In the characterisation of club Berkeleyness (Theorem \ref{thm:club-Berkeley}), can we equivalently consider $\sln_{\mu \omega}$?
\end{question}

Next to exacting cardinals, Aguilera, Bagaria, and Lücke \cite{aguilera2024large} also introduced  \emph{ultraexacting} cardinals. 
\begin{question}
	What is a compactness characterisation of ultraexacting cardinals?
\end{question}

	\bibliography{bibliography}{}
	\bibliographystyle{abbrv}

\end{document}